\documentclass[11pt]{article}

\usepackage[letterpaper,margin=0.8in]{geometry}
\usepackage[utf8]{inputenc}
\usepackage[T1]{fontenc}
\usepackage[english]{babel}

\usepackage{silence}
\usepackage[babel]{microtype}
\usepackage[dvipsnames]{xcolor}

\usepackage{csquotes}

\usepackage{amsmath,amssymb,amsthm}%,amsrefs}

\usepackage{mathrsfs}

\usepackage{mathtools}
\usepackage{commath}

\usepackage{thmtools}
\usepackage{thm-restate}

\usepackage{bbm}
\usepackage{dsfont}
\usepackage{lmodern}
\usepackage{fixcmex}

\usepackage{cases}
\usepackage{enumitem}
\setlist[enumerate,1]{label*=\arabic*., leftmargin=2em}

\usepackage{centernot}
\usepackage{booktabs}
\usepackage{multirow}
\usepackage{comment}
\usepackage{float}
\usepackage{scalerel}

\usepackage[pdfborderstyle={/S/U/W 0},unicode]{hyperref}
\hypersetup{
    colorlinks=true,
    linkcolor=blue!85!black,
    citecolor=blue!85!black,
    urlcolor=blue!85!black,
}

\usepackage[capitalise]{cleveref}

\usepackage{etoolbox}

\usepackage{subfiles}
\usepackage{pdfpages}

\ifdefined\checkunusedreferences
\usepackage{refcheck}

\makeatletter
\newcommand{\alsocheck}[1]{%
  \expandafter\let\csname @@\string#1\endcsname#1%
  \expandafter\DeclareRobustCommand\csname relax\string#1\endcsname[1]{%
    \csname @@\string#1\endcsname{##1}\@for\@temp:=##1\do{\wrtusdrf{\@temp}\wrtusdrf{{\@temp}}}}%
  \expandafter\let\expandafter#1\csname relax\string#1\endcsname
}
\newcommand{\alsocheckrange}[1]{%
  \expandafter\let\csname @@\string#1\endcsname#1%
  \expandafter\DeclareRobustCommand\csname relax\string#1\endcsname[2]{%
    \csname @@\string#1\endcsname{##1}{##2}\wrtusdrf{##1}\wrtusdrf{{##1}}\wrtusdrf{##2}\wrtusdrf{{##2}}}%
  \expandafter\let\expandafter#1\csname relax\string#1\endcsname
}
\makeatother

\alsocheck{\cref}
\alsocheck{\Cref}
\alsocheckrange{\crefrange}
\alsocheckrange{\Crefrange}

\else
\fi

\declaretheoremstyle[
  headfont=\normalfont\bfseries,
  bodyfont=\itshape,
  spaceabove=6pt,
  spacebelow=6pt,
  postheadspace=0.5em,
]{plain}

\declaretheoremstyle[
  headfont=\normalfont\bfseries,
  bodyfont=\normalfont,
  spaceabove=6pt,
  spacebelow=6pt,
  postheadspace=0.5em,
]{noital}

\declaretheorem[style=plain,name=Theorem,numberwithin=section]{theorem}

\declaretheorem[style=plain,sibling=theorem,name=Lemma]{lemma}

\declaretheorem[style=plain,sibling=theorem,name=Corollary]{corollary}

\declaretheorem[style=plain,numbered=no,name=Theorem]{theoremn-n}
\declaretheorem[style=plain,numbered=no,name=Theorems]{theorems-n}
\declaretheorem[style=plain,numbered=no,name=Proposition]{proposition-n}
\declaretheorem[style=plain,numbered=no,name=Propositions]{propositions-n}
\declaretheorem[style=plain,numbered=no,name=Lemma]{lemma-n}
\declaretheorem[style=plain,numbered=no,name=Lemmas]{lemmas-n}
\declaretheorem[style=plain,numbered=no,name=Corollary]{corollary-n}
\declaretheorem[style=plain,numbered=no,name=Corollaries]{corollaries-n}
\declaretheorem[style=plain,numbered=no,name=Conjecture]{conjecture-n}
\declaretheorem[style=plain,numbered=no,name=Conjectures]{conjectures-n}
\declaretheorem[style=plain,numbered=no,name=Question]{question-n}
\declaretheorem[style=plain,numbered=no,name=Questions]{questions-n}
\declaretheorem[style=plain,numbered=no,name=Claim]{claim-n}
\declaretheorem[style=plain,numbered=no,name=Claims]{claims-n}
\declaretheorem[style=plain,numbered=no,name=Fact]{fact-n}
\declaretheorem[style=plain,numbered=no,name=Facts]{facts-n}
\declaretheorem[style=plain,numbered=no,name=Problem]{problem-n}
\declaretheorem[style=plain,numbered=no,name=Problems]{problems-n}
\declaretheorem[style=plain,numbered=no,name=Open Problem]{openproblem-n}
\declaretheorem[style=plain,numbered=no,name=Open Problems]{openproblems-n}
\declaretheorem[style=plain,numbered=no,name=Challenge]{challenge-n}
\declaretheorem[style=plain,numbered=no,name=Challenges]{challenges-n}
\declaretheorem[style=plain,numbered=no,name=Exercise]{exercise-n}
\declaretheorem[style=plain,numbered=no,name=Exercises]{exercises-n}
\declaretheorem[style=plain,numbered=no,name=Property]{property-n}
\declaretheorem[style=plain,numbered=no,name=Properties]{properties-n}

\declaretheorem[style=noital,sibling=theorem,name=Remark]{remark}

\declaretheorem[style=noital,sibling=theorem,name=Definition]{definition}

\declaretheorem[style=noital,numbered=no,name=Remark]{remark-n}
\declaretheorem[style=noital,numbered=no,name=Remarks]{remarks-n}
\declaretheorem[style=noital,numbered=no,name=Definition]{definition-n}
\declaretheorem[style=noital,numbered=no,name=Definitions]{definitions-n}
\declaretheorem[style=noital,numbered=no,name=Construction]{construction-n}
\declaretheorem[style=noital,numbered=no,name=Constructions]{constructions-n}
\declaretheorem[style=noital,numbered=no,name=Observation]{observation-n}
\declaretheorem[style=noital,numbered=no,name=Observations]{observations-n}
\declaretheorem[style=noital,numbered=no,name=Example]{example-n}
\declaretheorem[style=noital,numbered=no,name=Examples]{examples-n}

\newcommand{\defined}{\mathrel{\coloneqq}}
\renewcommand{\leq}{\leqslant}

\renewcommand{\geq}{\geqslant}

\let\oldexists\exists
\let\exists\relax
\DeclareMathOperator{\exists}{\:\!\oldexists}
\let\oldforall\forall
\let\forall\relax
\DeclareMathOperator{\forall}{\:\!\oldforall}

\newcommand{\st}{\mathbin{\colon}}
\undef{\set}
\DeclarePairedDelimiter{\set}{\lbrace}{\rbrace}

\undef{\emptyset}
\newcommand{\emptyset}{\varnothing}

\undef{\mod}
\newcommand{\mod}[1]{\ (\mathrm{mod}\ #1)}

\undef{\abs}
\DeclarePairedDelimiterX{\abs}[1]
  {\lvert}{\rvert}{\ifblank{#1}{\,\cdot\,}{#1}}
\undef{\norm}
\DeclarePairedDelimiterX{\norm}[1]
  {\lVert}{\rVert}{\ifblank{#1}{\,\cdot\,}{#1}}
\DeclarePairedDelimiterX{\inner}[2]
  {\langle}{\rangle}{\ifblank{#1}{\,\cdot\,}{#1},\ifblank{#2}{\,\cdot\,}{#2}}

\DeclareMathOperator{\supp}{supp}

\DeclareMathDelimiter{\given}
  {\mathbin}{symbols}{"6A}{largesymbols}{"0C}
\DeclareMathOperator{\Prob}{\mathbb{P}}
\DeclarePairedDelimiterXPP{\prob}[1]
  {\Prob}{\lparen}{\rparen}{}
  {\renewcommand{\given}{\nonscript\;\delimsize\vert\nonscript\;\mathopen{}}#1}
\DeclareMathOperator{\Expec}{\mathbb{E}}
\DeclarePairedDelimiterXPP{\expec}[1]
  {\Expec}{\lparen}{\rparen}{}
  {\renewcommand{\given}{\nonscript\;\delimsize\vert\nonscript\;\mathopen{}}#1}
\DeclareMathOperator{\Var}{Var}
\DeclarePairedDelimiterXPP{\var}[1]
  {\Var}{\lparen}{\rparen}{}
  {\renewcommand{\given}{\nonscript\;\delimsize\vert\nonscript\;\mathopen{}}#1}
\DeclareMathOperator{\Cov}{Cov}
\DeclarePairedDelimiterXPP{\cov}[2]
  {\Cov}{\lparen}{\rparen}{}{#1,#2}

\newcommand{\NN}{\mathbb{N}}

\newcommand{\RR}{\mathbb{R}}

\newcommand{\cB}{\mathcal{B}}
\newcommand{\cC}{\mathcal{C}}
\newcommand{\cD}{\mathcal{D}}

\newcommand{\cF}{\mathcal{F}}
\newcommand{\cG}{\mathcal{G}}

\newcommand{\cI}{\mathcal{I}}

\newcommand{\cP}{\mathcal{P}}

\newcommand{\cT}{\mathcal{T}}

\newcommand{\cV}{\mathcal{V}}
\newcommand{\cW}{\mathcal{W}}

\newcommand{\rM}{\mathrm{M}}

\usepackage{tikz}
\usetikzlibrary{arrows, arrows.meta, backgrounds}
\usepackage{pgfplots}
\pgfplotsset{compat=1.18}

\usepackage{cancel}

\usepackage[
  backend=biber,
  style=alphabetic,
  maxnames=99,
  giveninits=true,
  doi=false,
  isbn=false,
  url=false
]{biblatex}
\NewDocumentCommand{\keq}{g g}{%
  \IfNoValueTF{#1}
    {\overset{k}{\sim}}
    {%
      \IfNoValueTF{#2}
        {\overset{#1}{\sim}}
        {\overset{#1,#2}{\sim}}%
    }%
}

\newcommand{\lp}[1]%{}
{\textcolor{red}{\textbf{[LP:} #1\textbf{]}}}
\newcommand{\lpoptional}[1]%{}
{\textcolor{red}{\textbf{[LP(Optional):} #1\textbf{]}}}

\title{Deck Reconstruction for Restricted Word Classes}

\author{%
  Krishna Angal\thanks{Department of Mathematics, University of Kentucky,
  Lexington, KY 40506. Email: \href{mailto:aean231@uky.edu}{aean231@uky.edu}.}%
  \hspace{2.1em}%
  Ruben Ascoli\thanks{School of Mathematics, Georgia Institute of Technology,
  Atlanta, GA 30332. Email: \href{mailto:rascoli3@gatech.edu}{rascoli3@gatech.edu}.}%
  \hspace{2.1em}%
  Simon X. He\thanks{Department of Mathematics, University of California,
  Berkeley, Berkeley, CA 94720. Email: \href{mailto:simonhe0727@berkeley.edu}{simonhe0727@berkeley.edu}.}%
  \hspace{2.1em}%
  Xiaoyu He\thanks{School of Mathematics, Georgia Institute of Technology,
  Atlanta, GA 30332. Email: \href{mailto:xhe399@gatech.edu}{xhe399@gatech.edu}.}%
  \\[0.65em]
  Logan Post\thanks{School of Mathematics, Georgia Institute of Technology,
  Atlanta, GA 30332. Email: \href{mailto:lpost3@gatech.edu}{lpost3@gatech.edu}.}%
  \hspace{2.6em}%
  Arker Oke Soe\thanks{Department of Mathematics, Cornell University,
  Ithaca, NY, USA. Email: \href{mailto:as3738@cornell.edu}{as3738@cornell.edu}.}%
  \hspace{2.6em}%
  Olga Zhukova\thanks{Department of Mathematics, Cornell University,
  Ithaca, NY, USA. Email: \href{mailto:ovz3@cornell.edu}{ovz3@cornell.edu}.}%
}

\date{July 2026}

\begin{document}

\renewcommand{\thefootnote}{\fnsymbol{footnote}}
\maketitle

\begin{abstract}
For a binary word $x\in \{0,1\}^+$, the $k$-deck of $x$ is the multiset of all length-$k$ subsequences counted with multiplicity. The deck reconstruction problem asks for the minimum $k$ such that every word of a fixed length is uniquely determined by its $k$-deck. In this paper, we study deck reconstruction for several restricted classes of binary words, including words with exactly $s$ ones, words with at most $\ell$ runs of ones, and words where the indices of ones are given by a monotone polynomial of degree $d$. We prove that for these classes of words, their reconstruction thresholds are bounded by a function of the structural complexity of the class.

We also study weak deck reconstruction, in which subwords may contain ``wildcard'' symbols that map to either 0 or 1. By generalizing a “template lemma” of Dudík and Schulman, we improve the upper bound in an open question about weak reconstruction from Dudík and Schulman. Finally, for the class of periodic words with period $d$, we prove that its in-class reconstruction threshold is equal to the general reconstruction threshold for words of length $d$, and give a general template construction producing further classes with the same threshold-preservation phenomenon.
\end{abstract}

\setcounter{footnote}{0}
\renewcommand{\thefootnote}{\arabic{footnote}}

%----------------------------------------------------------%
\begin{section}{Introduction}

For a fixed alphabet $\Lambda$, let $x \in \Lambda^n$ be a word of length $n$. For $k\leq n$, a word $w\in \Lambda^k$ is a \emph{subword} of $x$ if there is a strictly increasing map $s: [k] \to [n]$ such that $w_i = x_{s(i)}$ for all $i \in [k]$. The \emph{multiplicity} of $w$ in $x$, denoted $N(w, x)$, is the number of distinct such maps $s$. The \emph{$k$-deck} of $x$, denoted $\cD_k(x)$, is the multiset of all length-$k$ subwords of $x$ written with multiplicity; that is, $\cD_k(x) = \{(w, N(w, x)): w\in \Lambda^k\}$. For two words $x, y$, we say that $x \keq y$ if $\cD_k(x) = \cD_k(y)$. The deck reconstruction problem, first posed by Kalashnik \cite{Kalashnik1973} and further studied by Manvel et al. \cite{Manvel1991-mh}, asks for the smallest value of $k$ for which any $x \in \Lambda^n$ is uniquely determined by $\cD_k(x)$. More precisely, the deck reconstruction problem concerns the asymptotic behavior of the so-called ``reconstruction threshold''
\[r(n) \defined \min \set{k \st \text{no distinct } x,y \in \Lambda^n \text{ satisfy } x \keq y }.\]

As shown in \cite[Lemma 2]{Manvel1991-mh}, $r(n)$ is independent of the alphabet as long as it has size at least $2$. Thus, for the remainder of the paper we consider only binary words, so $\Lambda=\{0,1\}$. In \cite{Manvel1991-mh}, they give the first nontrivial bounds of $\log_{2}n \leq r(n) \leq \lfloor n/2 \rfloor$. The upper bound was later improved to $O(\sqrt{n\log n})$ by Scott \cite{Scott1997-mh}, and the current best known bounds are 
\[e^{\Omega (\sqrt{\log n})} \leq r(n) \leq O(\sqrt{n}),\]
where the upper bound was given by Krasikov and Roditty \cite{Krasikov1997-rp}, and the lower bound was given by Dud\'ik and Schulman \cite{Dudik2003-fz}.

The study of deck reconstruction is related to the Reconstruction Conjecture for graphs, first posed by P. J. Kelly and S. M. Ulam in 1941 \cite{Kelly1942}, which asserts that every finite simple graph on at least three vertices is reconstructible from its vertex-deleted subgraphs. This conjecture has been studied for a variety of graph classes, including disconnected graphs, trees, and regular graphs \cite{Kelly1957}. Other reconstruction problems such as trace reconstruction \cite{Batu2004-pq, Chase2020-if} and permutation reconstruction \cite{Raykova2006-jl, Ma2024-sk}, along with variants of the deck reconstruction problem like gapped deck reconstruction \cite{Golm2022-jy}, reconstruction of multiple strings simultaneously \cite{Yehezkeally2023-jm}, and reconstruction given long and short subwords \cite{Gabrys2017-ju} have also been studied.

Motivated by the work on the Reconstruction Conjecture for particular graph classes, we aim to study the deck reconstruction problem for specific classes of words. We show that restricted to specific classes of words, the reconstruction threshold is no longer a function of the word's length, but rather of the structural complexity of the class.

Let $\set{0,1}^+$ denote the set of all non-empty finite-length binary strings. Then, for a given class of words $\cC \subseteq \set{0,1}^+$, we define the \textit{reconstruction threshold}
\[r(\cC) \defined \min\set{k \st \text{no } x \in \cC \text{ and }y \in \{0,1\}^+ \text{ with } x \neq y \text{ satisfy }x \keq y}.\]
For some classes of words, we also study the \textit{in-class reconstruction threshold} 
\[r'(\cC) \defined \min\set{k \st \text{no distinct } x,y \in \cC \text{ satisfy } x \keq y }. \] 
Trivially, we have $r'(\cC) \leq r(\cC)$: the $r(\cC)$-deck suffices to distinguish any word in $\cC$ from \emph{any} other binary word, while the $r'(\cC)$-deck only promises to distinguish any word in $\cC$ from any other word in $\cC$.

We next introduce the specific word classes we study and the results obtained for each class. Let $[n] \coloneqq \{1, 2, \dots, n\}$. Let $\supp(x) \defined \{i\in [n] : x_i = 1\}$ be the set of indices where $x$ has a $1$ entry.

For
\[
\cW_s \defined \{x\in\{0,1\}^+:|\supp(x)|=s\},
\]
which is the class of words with exactly $s$ ones, we obtain the following.

%--------------------------------%
\begin{restatable}{theorem}{thmfixedweight}
\label{thm:fixedweight}
For $s \geq 5$, we have 
\[r(\cW_s) \leq s.\]
\end{restatable}
%--------------------------------%

 Define the class
\[ 
\cP_d \defined \set{x\in \cW_s : s\geq 1,\ \text{$\supp(x)$ is the image of a degree $d$ polynomial monotonic on $[s]$ applied to $[s]$}},
\]
for which we obtain the following bound.

%--------------------------------%
\begin{restatable}{theorem}{thmpolysupport}
\label{thm:polynomialsupport}
For $d \geq 0$, we have 
\[r(\cP_d) \leq d+2.\]
\end{restatable}
%--------------------------------%
Note that $\cW_s\subseteq \cP_{s-1}$. Next, we study
\[
\cB_\ell \defined \set{ x : x \text{ is a binary word with at most } \ell \text{ runs of 1s}},
\]
and obtain the following result.

%--------------------------------%
\begin{restatable}{theorem}{thmbddruns}
\label{thm:bdd-runs}
For $\ell \geq 1$, we have
\[
\exp\left(\Omega(\sqrt{\log 2\ell})\right) \leq r'(\cB_\ell) \leq 2\ell,
\]
and
\[
r'(\cB_\ell)\leq r(\cB_\ell) \leq 2\ell + 1.
\]
\end{restatable}
%--------------------------------%

We also provide a more general way to obtain bounds on the reconstruction thresholds for certain word classes by studying ``templates'', which were previously studied by \cite{Manvel1991-mh, Dudik2003-fz}. Specifically, we generalize a ``template lemma'' presented in \cite[Lemma 9]{Dudik2003-fz} to provide a better bound for a weaker form of word reconstruction. 

A \textit{weak subword} of $x\in \{0,1\}^n$ is a word $w\in \{0,1,\star\}^k$ together with a strictly increasing map $s:[k] \to [n]$ such that 
\[
x_{s(i)} =
\begin{cases}
0 \quad \text{ if } w_i = 0, \\
1 \quad \text{ if } w_i = 1, \\
0 \text{ or } 1 \quad \text{ if } w_i = \star.
\end{cases}
\]
We call $\star$ a ``wildcard'' letter since it can map to any letter of $x$. Then, the $(k, t)$-deck of $x$, denoted $\cD_{k, t}(x)$, is the multiset of all subwords of $x$ with length $k$ and $t$ \textbf{non}-wildcards (hence $k-t$ wildcards). For example, the $(2, 1)$-deck of $0111$ is 
\[\{(\star1, 6), (1\star, 3), (\star0, 0), (0\star, 3)\}.\]

For words $x, y$, we say $x \keq{k}{t} y$ if $\cD_{k, t}(x) = \cD_{k, t}(y)$, and we define the \textit{weak reconstruction threshold}
\[
r(n, t) \defined \min \set{k \st 
\text{no distinct } x,y \in \{0, 1\}^n \text{ satisfy } x \keq{k}{t} y }.
\]
We also let
\[
S(k, t) 
:= 
\min\{n \st \text{there exist distinct } x, y \in \{0, 1\}^n \text{ such that } x \keq y\} = \min\{n:r(n,t)>k\}.
\]
Clearly bounds on $S(k,t)$ translate directly to bounds on $r(n,t)$, but $S(k,t)$ is easier to work with in the proof of the following theorem.
Our generalized template lemma (\Cref{lem:template-lem}) gives the following improvement, which also answers parts of open question 2 in \cite{Dudik2003-fz}.

%--------------------------------%
\begin{restatable}{theorem}{thmweakdeckupper}\label{thm:weak-deck-upper}
    For $k \to \infty$ and for any $0 \leq t \leq k$, 
    \[
    S(k,t) \leq k^{\frac{3}{\log 3} \log (t+1) + O( \log \log k )}.
    \]
\end{restatable}
%--------------------------------%
Note that this result generalizes the bound obtained by Dud\'ik and Schulman \cite{Dudik2003-fz}, who proved it for $t=k$. Roughly speaking, we use a similar main idea from their proof while noticing that their condition on the substituted pairs for the template can be weakened to provide more flexibility.

We use similar machinery to get a conditional template that can be used to generate other classes of words that preserve the reconstruction thresholds for in-class reconstruction.

For a \textit{template} word $p \in \{0, 1\}^+$ and \textit{substitution} words $x, y \in \{0, 1\}^+$, define $p(x, y) \in \{0, 1\}^+$ as the word obtained by replacing each $0$ in $p$ by $x$ and each $1$ by $y$. 

%--------------------------------%
\begin{restatable}{theorem}{thmtemplateequivalence}\label{thm:template-equivalence}
Let $p, q, x, y \in \{0, 1\}^+$. For $k \geq 1$, it holds that
\begin{enumerate}
    \item[(1)] If $\frac{N(0,x)}{N(1, x)} \neq \frac{N(0,y)}{N(1,y)}$, then
    \[
    p \keq q \iff p(x, y) \keq q(x, y).
    \]

    \item[(2)] If $|p| = |q|$ and $N(1,p) \neq N(1,q)$, then
    \[
    x \keq y \iff p(x, y) \keq q(x, y).
    \]
\end{enumerate}
\end{restatable}

%--------------------------------%

As an example, we can apply this result to the class
\[
\cT_d \defined \set{ x : x \text{ is a periodic binary word with period } d},
\]
 to obtain the following.

%--------------------------------%
\begin{restatable}{corollary}{corperiodic}\label{cor:periodic}
For $d \geq 1$, we have
\[
r'(\cT_d) = r(d).
\]
\end{restatable}

%--------------------------------%

The remainder of the paper is organized as follows. In \Cref{sec:moment-interpretation}, we give an alternative characterization of $(k, t)$-equivalence that was implicit in \cite{Dudik2003-fz}. In  \Cref{subsec:fixedweight}, \Cref{subsec:polynomialsupport}, and \Cref{subsec:boundedruns}, we prove Theorems \ref{thm:fixedweight}, \ref{thm:polynomialsupport}, and \ref{thm:bdd-runs}, respectively. In \Cref{sec:weak-deck-bound}, we prove Theorem \ref{thm:weak-deck-upper} using the generalized template lemma. Finally, in \Cref{sec:template-classes}, we prove Theorem \ref{thm:template-equivalence} through a counting and induction argument and show Corollary \ref{cor:periodic}.
\end{section}
%----------------------------------------------------------%

%----------------------------------------------------------%
\begin{section}{Moment Interpretation of Weak Deck Equivalence}\label{sec:moment-interpretation}

We first present a characterization of $(k,t)$-equivalence which was implicitly noted by Dud\'ik and Schulman \cite{Dudik2003-fz}.

%--------------------------------%
\begin{lemma}[Moment interpretation]
\label{lem:momentinterpretation} Let $x,y$ be length $n$ binary words. Then, $x \keq{k}{t} y$ if and only if for all $\ell \leq t,$ and $a_1 + a_2 + \dots + a_\ell \leq k -\ell$, we have 
\[\sum_{1 \leq i_1 < i_2 < \dots < i_\ell \leq n} i_1^{a_1} i_2^{a_2} \dots i_\ell^{a_\ell} x_{i_1} x_{i_2} \dots x_{i_\ell} = \sum_{1 \leq i_1 < i_2 < \dots < i_\ell \leq n} i_1^{a_1} i_2^{a_2} \dots i_\ell^{a_\ell} y_{i_1} y_{i_2} \dots y_{i_\ell}. \]
\end{lemma}
\begin{proof}
For the forward direction, assume $x \keq{k}{t} y $. It is easily seen that this implies that $x \keq{k'}{\ell} y$ for any $k'\leq k, \ell\leq t$. Thus, 
\[
N(\star^{g_1} 1 \star^{g_2}   \dots \star^{g_\ell} 1 , x )
=
N(\star^{g_1} 1 \star^{g_2} \dots \star^{g_\ell} 1 , y)
\] 
whenever $g_1 + g_2 + \dots + g_\ell \leq k - \ell$, where $\star^{n}$ denotes a run of $n$ wildcards. Notice that
\[
N(\star^{g_1} 1 \star^{g_2} \dots \star^{g_\ell} 1,x)
=
\sum_{1\leq i_1<\cdots<i_\ell\leq n}
\binom{i_1-1}{g_1}\binom{i_2-i_1-1}{g_2}\cdots \binom{i_\ell-i_{\ell-1}-1}{g_\ell}
x_{i_1}\cdots x_{i_\ell}.
\]
Now make the change of variables
\[
u_1=i_1-1,\qquad u_2=i_2-i_1-1,\qquad \dots,\qquad u_\ell=i_\ell-i_{\ell-1}-1.
\]
The polynomials
\[
\binom{u_1}{g_1}\binom{u_2}{g_2}\cdots\binom{u_\ell}{g_\ell},
\qquad g_1+\cdots+g_\ell\leq k-\ell,
\]
span all polynomials in $u_1,\dots,u_\ell$ of total degree at most $k-\ell$. Hence, equality of the above counts for all $g_1+\cdots+g_\ell\leq k-\ell$ implies that for every polynomial $P$ of total degree at most $k-\ell$, we have
\[
\sum_{1\leq i_1<\cdots<i_\ell\leq n}
P(i_1,\dots,i_\ell)x_{i_1}\cdots x_{i_\ell}
=
\sum_{1\leq i_1<\cdots<i_\ell\leq n}
P(i_1,\dots,i_\ell)y_{i_1}\cdots y_{i_\ell}.
\]
Taking $P(i_1,\dots,i_\ell)=i_1^{a_1}\cdots i_\ell^{a_\ell}$
with $a_1+\cdots+a_\ell\leq k-\ell$ gives the forward direction.

For the other direction,  suppose for all $\ell \leq t,$ and $a_1 + a_2 + \dots + a_\ell \leq k -\ell$, we have 
\[\sum_{1 \leq i_1 < i_2 < \dots < i_\ell \leq n} i_1^{a_1} i_2^{a_2} \dots i_\ell^{a_\ell} x_{i_1} x_{i_2} \dots x_{i_\ell} = \sum_{1 \leq i_1 < i_2 < \dots < i_\ell \leq n} i_1^{a_1} i_2^{a_2} \dots i_\ell^{a_\ell} y_{i_1} y_{i_2} \dots y_{i_\ell}. \]
Then, for any word $\star^{g_0}  1 \star^{g_1} 1 \dots 1 \star^{g_\ell} $ with $\ell \leq t$ and $g_0 + g_1 + \dots + g_\ell \leq k-\ell$, we know 
\[N(\star^{g_0}  1 \star^{g_1} 1 \dots 1 \star^{g_\ell} ,x) = \sum_{1 \leq i_1 < \dots <i_\ell \leq n } {i_1 -1 \choose g_0} {i_2 - i_1 - 1 \choose g_1} \dots {n - i_\ell \choose g_\ell} x_{i_1}x_{i_2} \dots x_{i_\ell}.\]
Since the weight is a polynomial in $(i_j)_j$ with degree at most $k-\ell$, the assumption immediately gives 
\[\sum_{1 \leq i_1 < \dots <i_\ell \leq n } {i_1 -1 \choose g_0} {i_2 - i_1 - 1 \choose g_1} \dots {n - i_\ell \choose {g_\ell}} x_{i_1}x_{i_2} \dots x_{i_\ell} \]
\[ = \sum_{1 \leq i_1 < \dots <i_\ell \leq n } {i_1 -1 \choose g_0} {i_2 - i_1 - 1 \choose g_1} \dots {n - i_\ell \choose {g_\ell}} y_{i_1}y_{i_2} \dots y_{i_\ell} = N(\star^{g_0}  1 \star^{g_1} 1 \dots 1 \star^{{g_\ell}},y).\]
Notice that 
\[
N(\star^{g_0}  \overset{{\downarrow}}{1} \star^{g_1} 1 \dots 1 \star^{g_\ell},y) + N(\star^{g_0}  \overset{\downarrow}{0} \star^{g_1} 1 \dots 1 \star^{g_\ell},y) = N(\star^{g_0}  \overset{\downarrow}{\star} \star^{g_1} 1 \dots 1 \star^{g_\ell},y).
\]
Repeatedly applying
$N(u0v,y)=N(u\star v,y)-N(u1v,y)$ eliminates one zero at each
step, and hence, $N(w,y)$ can be written as a linear combination of counts
of length-$k$ words containing only wildcards and at most $t$ ones.
\end{proof}
%--------------------------------%

For ease of working with this characterization of $(k,t)$-equivalence, we set the following definition.
%--------------------------------%
\begin{definition} 
Let $x$ be a binary word of length $n$. Then for non-negative integers $a_1, a_2, \dots ,a_\ell$ with $\ell \leq n$, we define \textit{the moment of $x$ at $(a_1, a_2, \dots, a_\ell)$} to be
\[\rM_{(a_1 , \dots, a_\ell) } (x) = \sum_{1 \leq i_1 < \dots < i_\ell \leq n } i_1^{a_1} \dots i_r^{a_\ell} x_{i_1} \dots x_{i_\ell}.\]
\end{definition}
%--------------------------------%

\end{section}
%----------------------------------------------------------%

%----------------------------------------------------------%
\begin{section}{Proofs for Restricted Classes}
\label{sec:proofsforrestrictedclasses}
\subsection{Fixed Number of Ones}
\label{subsec:fixedweight}

In this section, we study the reconstruction threshold for words with a fixed number of 1s.

\thmfixedweight*
\begin{proof}%[Proof of \cref{thm:fixedweight}]
Let $s \geq 5$ and $x \in \cW_s,y \in \set{0,1}^+$  such that $\cD_s(x) = \cD_s(y)$. Note $N(1,x)=N(1,y)=s$. We will show that $x = y$. First, define $g=(g_i)_i$ by letting $x=0^{g_0} 1 0^{g_1} 1 \ldots 1 0^{g_s}$, and similarly define $g'=(g_i')_i$ by $y=0^{g_0'} 1 0^{g_1'} 1 \ldots 1 0^{g_s'}$.

Let $\delta \defined g' - g$. We need to show $\delta = 0$. 

We begin by considering the following word count for $r \leq s -1$:
\[N(1^r 0 1^{s-1 - r} ,x) = \sum_{i = 0}^s g_i {i \choose r} { s - i \choose s-1 -r}%=g_r\binom{r}{r}\binom{s-r}{s-1-r}+g_{r+1}\binom{r+1}r\binom{s-r-1}{s-1-r},
=g_r(s-r)+g_{r+1}(r+1),
\]
since the binomial coefficients are only nonzero when $i = r, r+ 1$. Equality of these word counts gives 
\[(s - r)\delta_r + (r+1)\delta_{r+1}= 0 \] for each $r \leq s-1$.
From this, we recursively compute that
\begin{equation}
\label{eq:gapidentity}
g'_i = g_i + \delta_0(-1)^i {s \choose i}.
\end{equation}
Thus, if $\delta_0 = 0$, we are done. If not, consider 
\[N(001^{s-2} ,x) = {s -1 \choose 2}{g_0 \choose 2} + (s-2){g_0 +g_1 \choose 2}  + {g_0 + g_1 + g_2 \choose 2}, \]
and the analogous holds for $y$. Substituting \eqref{eq:gapidentity} into the equality $N(001^{s-2},x) = N(001^{s-2},y)$ with the above expansion gives
\begin{equation}
\label{eq:g2identity}
g_2 = g_1 - \frac{s(s+1)}{4} \delta_0.
\end{equation}
The computations behind this equality are somewhat involved, so we present them in the Appendix~\ref{sec:appendix}. Next, by considering the position of the first 1, we have
\[N(001^{s-3},x) = \sum_{i=0}^3{s-i-1 \choose 3-i} {g_{0} + g_{1} + \dots + g_{i} \choose 2}.\]
Equality of these counts together with \eqref{eq:gapidentity} and \eqref{eq:g2identity} gives 
\begin{equation}
\label{eq:g3identity}
g_3 = g_1 + \frac{s(s+1) (s-4)}{12}\delta_0.
\end{equation}
Once again, we give a complete proof of this identity in \cref{sec:appendix}.
Finally, by again considering the position of the first 1,
\[N(0001^{s-3} ,x) = \sum_{i =0}^3 {s -i -1 \choose 3-i }{g_0 + g_1 +\dots +g_{i}  \choose 3 }.\]
Then, equality of these counts together with $\eqref{eq:gapidentity}$, $\eqref{eq:g2identity}$, and $\eqref{eq:g3identity}$ gives
\begin{equation}
\label{eq:delta0identity}
s(s-1)(s-2)(s-3)(s-4)(s+1)(s^3 - 9s^2 + 8s+6)  \delta_0^3 = 0.
\end{equation} Once again, details about how to obtain this equality are given in \cref{sec:appendix}.
Since $s^3 - 9s^2 + 8s+6$ has no integer roots greater than or equal to $5$, the quantity above can only vanish when $\delta_0 = 0$ for all $s \geq 5$. Hence, 
\[r(\cW_s) \leq s,\]
for all such $s$. 
\end{proof}

\begin{remark}
It can also be shown by extending this technique that $r(\cW_s) \leq s-1$ for $s \geq 7$, but we omit the proof in this paper. 
\end{remark}

\begin{remark}
 Note that the condition $s\geq 5$ is required as $10^5110^51 \keq{4} 01010^61010$.    
\end{remark}

We note an easy lower bound on $r(\cW_s)$. Since any distinct $k$-equivalent pair of length $2s$ can be complemented if necessary so that both words have less than $s$ ones, we have
\[e^{\Omega(\sqrt{\log2s})}\leq r(2s)\leq r(\cW_s). \] 

%----------------------------------------------------------%

%----------------------------------------------------------%
\subsection{Polynomial Support}
\label{subsec:polynomialsupport}

%--------------------------------%

Now, we study the reconstruction threshold for $\mathcal{P}_d$, the family of words with polynomial support. 

\thmpolysupport* 

\begin{proof}%[Proof of \cref{thm:polynomialsupport}]
    If $d = 0$, $|N(1, x)| = 1$, so $N(01, x)$ and $N(10, x)$ together determine the position of the $1$ in $x$, meaning that $x$ is $2$-reconstructible. 

    Henceforth, we only consider $d \geq 1$. Let $T = N(1, x)$. Since $x \in \cP_d$, there exists a degree $d$ polynomial $p(x) \in \RR[x]$ that is monotonic on $[T]$ such that $\supp(x) = \{p(1), \dots, p(T)\}$. Without loss of generality, we may assume $p$ is increasing on $[T]$.

    Suppose there exists some $y \in \{0,1\}^+$ such that $x \overset{d+2}\sim y$. Then,  $x \overset{1}\sim y$, which implies that $|x|=|y|$ and $N(1,y)=N(1,x) = T$. Let $f : [T] \rightarrow \NN$ be an increasing function determining the positions of the 1s in $y$. We view $p, f$ as vectors in $\RR^T$ and consider the difference $\delta = p - f$. We show that $\delta = 0$ by computing $\|\delta\|^2=\sum_t\delta_t^2$. 
    
    We first expand and rearrange $\|f\|^2 = \|p - \delta\|^2$ to obtain 
    \begin{equation}\label{eq:deltaidentity}
        \|\delta\|^2=\|f\|^2-\|p\|^2+2\inner{p}{\delta}.
    \end{equation}
    Because $p(t)$ is a degree $d$ polynomial in $\RR[x]$, it must be of the form $p(t) = \sum_{k = 0}^d c_kt^k$ for some $c_k \in \RR$. Thus, 
    \[
       \inner{p}{\delta} = \sum_{t = 1}^T \delta_t p(t) = \sum_{t = 1}^T \delta_t  \sum_{k = 0}^d c_kt^k = \sum_{k = 0}^d c_k \sum_{t = 1}^T t^k \delta_t. 
    \]

    Substituting into Equation \eqref{eq:deltaidentity}, we get  
    \begin{equation}\label{eq:deltaidentitysub}
        \|\delta\|^2=\|f\|^2-\|p\|^2+2\sum_{k = 0}^d c_k \sum_{t = 1}^T t^k \delta_t.
    \end{equation}
    Now, we prove that this is $0$. For all $a, k \geq 0$ with $a + k \leq d + 1$, we observe that for $A_k = (a, \overbrace{0, \dots 0}^k)$:
    \[
        M_{A_k}(x) = \sum_{i<j_1<\cdots<j_k}i^ax_i\cdot x_{j_1}\cdots x_{j_k} = \sum_{t=1}^T p(t)^a{T-t\choose k}.
    \]
    For each $k$ such that $a + k \leq d + 1$, the expression $T - t \choose k$ is a degree-$k$ polynomial in $t$ since $T$ is fixed. Hence, if we define
    \[
        s_{a,k}(x)\coloneqq \sum_{t=1}^Tt^kp(t)^a,
    \]
    we can compute its value using $M_{A_k}(x)$. Because $x \overset{d+2}\sim y$, Lemma \ref{lem:momentinterpretation} gives that $M_{A_k}(x) = M_{A_k}(y)$ for $a + k \leq d + 1$. Since the $s_{a, k}(x)$ and $s_{a, k}(y)$ values are uniquely determined by the $M_{A_k}(x)$ and $M_{A_k}(y)$ values respectively, this implies that $s_{a, k}(x) = s_{a, k}(y)$ for $a + k \leq d + 1$. In particular, when $a = 2$, $k = 0$, we have $\|p\|^2 = \|f\|^2$, and when $a = 1$, $k \leq d$, we have $\sum_t t^k\delta_t = 0$. Substituting into Equation \eqref{eq:deltaidentitysub}, we conclude $\|\delta\|^2 = 0$, which implies that $x = y$.
\end{proof}

%--------------------------------%
%----------------------------------------------------------%

%----------------------------------------------------------%
\subsection{Bounded Number of Runs}
\label{subsec:boundedruns}

For $w \in \{0, 1\}^n$, a \textit{run of 1s} is a maximal interval $[i, j] \subset [n]$ such that $w_i = w_{i+1} = \dots = w_j = 1$. Let $R(w)$ denote the number of runs of 1s in $w$, and recall that
\[
\cB_\ell \defined \set{ x \in \{0, 1\}^+ : R(x) \leq \ell}.
\]
For a sequence $\delta\in\{-1,0,1\}^n$, let $S(\delta)$ be the number of runs of either $-1$ or $1$, which we call signed runs. For example, if $\delta =(-1, 1,0,-1,0,-1,-1)$, then the corresponding intervals of the signed runs are $[1, 1], [2, 2], [4, 4],$ and $[6,7]$, so $S(\delta)=4$.

The following lemma gives control over the number of signed runs for the difference vector of two words in $\cB_\ell$.

%--------------------------------%
\begin{lemma}\label{lem:bounded-sign-number}
If $x, y \in \mathcal{B}_\ell$ and $\delta_i = x_i - y_i$, then 
$
S(\delta) \leq 2\ell.
$
\end{lemma}

\begin{proof}
For $w \in \{0, 1\}^n$, using the convention that $w_0 = w_{n+1} = 0$, we have
\[
2R(w) = \sum_{i = 0}^{n} |w_{i+1} - w_i|,
\]
because each run contributes 1 to the right-hand sum whenever it starts or ends. Similarly, we have
\[
2S(\delta) = \sum_{i = 0}^{n} |\delta_{i+1} - \delta_i|,
\]
as each signed run contributes 1 to the right-hand sum whenever it starts or ends (transitions from $1$ to $-1$ contribute 2 as the current 1-run ends and the next $-1$-run starts, and similarly for transitions from $-1$ to $1$). Therefore, since $x, y \in \cB_\ell$, we have
\begin{align*}
2S(\delta) 
&= 
\sum_{i = 0}^{n} |\delta_{i+1} - \delta_i| \\
&=
\sum_{i = 0}^{n} |(x_{i+1} - y_{i+1}) - (x_{i} - y_{i})| \\
&\leq
\sum_{i = 0}^{n} |x_{i+1} - x_{i}| + \sum_{i = 0}^{n} |y_{i+1} - y_{i}| \\
&\leq
4\ell.\qedhere
\end{align*}

\end{proof}
%--------------------------------%

The moment characterization in Lemma \ref{lem:momentinterpretation} gives algebraic constraints on words that are weakly equivalent. The next lemma connects these moment sums to the signed runs in the difference vector.

%--------------------------------%
\begin{lemma}\label{lem:sign-condition}
For $\delta \in \{-1, 0, 1\}^n$, let $s = S(\delta)$. If
\[
\sum_{i = 1}^{n} i^a \delta_i = 0\quad \text{ for all } a \leq s-1,
\]
then $\delta = 0$.
\end{lemma}

\begin{proof}
We first show that there exists a polynomial $F(t) \in \mathbb{R}[t]$ with $\deg F(t) \leq s-1$ such that $\delta_iF(i) > 0$ for any $i \in [n]$ with $\delta_i \neq 0$.

Order the corresponding intervals of the signed runs as $I_1 < I_2 < \cdots < I_s$, and denote $I_j = [a_j, b_j]$. Let $\sigma_j \in \{+1, -1\}$ be the sign of the run $\delta_{I_j}$. Choose roots $c_j =b_j+\frac 12$ having the property that  $b_j < c_j < a_{j+1}$, and define
\[
F(t) := \sigma_1\cdot\!\!\prod_{\substack{1 \leq j \leq s-1 \\ \sigma_j \neq \sigma_{j+1}}} (c_j-t).
\]
We see that $F(i)$ changes sign whenever $\delta_i$ changes sign for $i \in [n]$, and since $F(a_1)$ has sign $\sigma_1$ we have $\delta_iF(i) > 0$ whenever $\delta_i \neq 0$.

Now, since $\deg F(t) \leq s-1$, we can write $F(t) = \sum_{a = 0}^{s-1} \alpha_a t^a$ for $\alpha_0, \dots, \alpha_{s-1} \in \mathbb{R}$. By the assumption of the lemma, we have 
\[
\sum_{i = 1}^{n} i^a \delta_i = 0 \quad  \text{ for all } a \leq s-1,
\]
so
\[
\sum_{i = 1}^{n} \delta_iF(i) = \sum_{i = 1}^{n} \delta_i \sum_{a = 0}^{s-1} \alpha_a i^a = \sum_{a = 0}^{s-1} \alpha_a \sum_{i = 1}^{n} \delta_i i^a = 0,
\]
which (by construction) can only happen if $\delta = 0$.
\end{proof}

%--------------------------------%

With the two lemmas, we can now prove our bound on the reconstruction threshold of $\cB_\ell$.

%--------------------------------%
%\thmbddruns*

\begin{proof}[Proof of \cref{thm:bdd-runs}]
First suppose that $x, y \in \cB_\ell$, and let $\delta =x-y$. By Lemma \ref{lem:bounded-sign-number}, we have $s := S(\delta) \leq 2\ell$. For $k \geq 2\ell$, if $x \overset{k}{\sim} y$, then by \cref{lem:momentinterpretation}, we have $\sum_{i = 1}^{n} i^a \delta_i = 0$ for all $a \leq s-1$, but by Lemma \ref{lem:sign-condition}, this forces $x = y$. Since this holds for any $k \geq 2\ell$, we obtain the desired upper bound for $r'(\cB_\ell)$.

Now suppose that $y \in \{0, 1\}^+$. For any $k \geq 2\ell + 1$, we know that $x \in \cB_\ell \text{ if and only if } N(1(01)^\ell, x) = 0$. Now if $N(1(01)^\ell, y) \neq 0$, we know that $y \notin \cB_\ell$, which distinguishes $x$ from $y$. Otherwise, we reduced the case to both $x, y \in \cB_\ell$, which has upper bound less than $2\ell + 1$.

For the lower bound, observe that any $x \in \{0, 1\}^{2\ell}$ is automatically in $\cB_\ell$, and hence from Dud\'ik and Schulman's \cite{Dudik2003-fz} lower bound, we have
\[
\exp\left(\Omega(\sqrt{\log 2\ell})\right) \leq r(2\ell) \leq r'(\cB_\ell).
\qedhere \] 
\end{proof}
%--------------------------------%

\end{section}
%----------------------------------------------------------%

%----------------------------------------------------------%
\begin{section}{Template Lemma} \label{sec:template}

%----------------------------------------------------------%
\begin{subsection}{Weak Equivalence Lower Bound}\label{sec:weak-deck-bound}

Recall that 
\[
S(k, t) 
:= 
\min\{n : \exists x, y \in \{0, 1\}^n, \text{ with } x \neq y, x \keq{k}{t} y\},
\]
and for $p, x, y \in \{0, 1\}^+$, $p(x, y) \in \{0, 1\}^+$ is obtained by replacing each $0$ in $p$ with $x$ and each $1$ by $y$. The next lemma generalizes \cite[Lemma 9]{Dudik2003-fz}, where the idea of blowing up template pairs is used to obtain a stronger equivalence. 

%--------------------------------%
\begin{lemma}\label{lem:template-lem}
Let $0 \leq a \leq k$ and $0 \leq t$, $T \leq K$. Suppose that
\[
t \geq \max_{0 \leq j \leq T} \left\lfloor\frac{j}{a + 1}\right\rfloor + \left\lfloor\frac{K - j}{k + 1}\right\rfloor.
\]
Then, for any binary words $x, y, p, q \in \{0, 1\}^+$, we have
\[
    x \keq{k}{a} y \text{ and } p \keq{K}{t} q \implies p(x, y) \keq{K}{T} q(x, y).
\]
Thus,
\[
S(K, T) \leq S(k, a)S(K, t).
\]
\end{lemma}

\begin{proof}
Since $p \keq{K}{t} q$, we have $|p| = |q| =: l$. By definition, $p(x, y) = p_1(x, y)\dots p_l(x, y)$, where
\[
p_i(x, y) = 
\begin{cases}
x \text{ if } p_i = 0, \\
y \text{ if } p_i = 1.
\end{cases}
\]
We will call each $p_i(x, y)$ a \textit{blow-up block} for the remainder of the proof.

Let $w \in \{0, 1, \star\}^{\leq K}$ with at most $T$ non-wildcards, and for $m \leq |w|$, let $w = w^{(1)} \dots w^{(m)}$ be a partition of $w$ into $m$ nonempty consecutive blocks. A \textit{partition embedding} is a strictly increasing map $s:[m] \to [l]$, which specifies that block $w^{(i)}$ is embedded into the blow-up block $p_{s(i)}(x, y)$. Let $K_i = |w^{(i)}|$ and $T_i$ be the number of non-wildcards in $w^{(i)}$. Then, for each map $s$, we can associate a ``pattern'' $v \in \{0, 1, \star\}^m$, where
\[
v_i
=
\begin{cases}
p_{s(i)} \quad \text{if } K_i > k \text{ or } T_i > a, \\
\star \quad \text{ otherwise}.
\end{cases}
\]
Let $\cP_{w}$ be the set of all possible partitions of $w$ into nonempty consecutive blocks. For a fixed partition $\cP \in \cP_{w}$ and a fixed pattern $v\in \{0,1,\star\}^{|\cP|}$, define
\[
N'_{\cP, v}(w,x,y)
:=
\prod_{i:v_i=0}N(w^{(i)}, x) \prod_{i:v_i=1}N(w^{(i)}, y) \prod_{i:v_i=\star}N(w^{(i)}, x).
\]
That is, $N'_{\cP, v}(w,x,y)$ counts the number of ways to embed the blocks $w^{(1)}\dots w^{(|\cP|)}$ within the blow-up blocks under a fixed partition embedding $s$ associated with pattern $v$. Indeed, if $v_i = 0$, then $p_{s(i)}(x, y) = x$, so there are $N(w^{(i)}, x)$ choices. A similar argument justifies the second product. If $v_i = \star$, we know that $K_i \leq k$ and $T_i \leq a$. Since $x \keq{k}{a} y$, we have $N(w^{(i)}, x)$ choices regardless of whether $p_{s(i)}(x, y) = x$ or $y$. Since embeddings in each block is independent, multiplying over all the choices yields the desired count.

For a fixed partition $\cP$, define $\cV(\cP) := \{v \in \{0, 1, \star\}^{|\cP|} : v_i = \star \iff K_i \leq k \text{ and } T_i \leq a\}$. Then, we have the counting formula
\begin{equation}\label{eq:template-count}
N(w, p(x, y))
=
\sum_{\cP \in \cP_{w}} \sum_{v \in \cV(\cP)} N'_{\cP, v}(w,x,y) \, N(v, p).
\end{equation}
For each $\cP$, $\cV(\cP)$ contains every admissible pattern that can arise from a partition embedding $s$. Then, for each $v \in \cV(\cP)$, $N(v, p)$ is exactly the number of partition embeddings $s$ satisfying $p_{s(i)} = v_i$ with $v_i \neq \star$. Therefore, summing over all the possible partitions and corresponding choices for the pattern gives $N(w, p(x, y))$.

We can write $N(w, q(x, y))$ similarly by replacing $p$ with $q$ in \eqref{eq:template-count}. So, to conclude that $N(w, p(x, y)) = N(w, q(x, y))$, it remains to show that $N(v, p) = N(v, q)$. That is, for each $\cP \in \cP_w$, the pattern $v$, associated with any partition embedding $s$, has length at most $K$ and at most $t$ non-wildcards. 

For a partition $\cP$ and pattern $v$, let $\cI_1 = \{i : T_i > a\}$ and $\cI_2 = \{i : K_i > k \text{ and } T_i \leq a\}$. Then for $j := \sum_{i \in \cI_1} T_i\geq |\cI_1|(a+1)$, we have 
\[
|\cI_1| \leq \Big\lfloor \frac{j}{a+1} \Big\rfloor.
\]
Now since $K_i\geq T_i$ for all $i$, we have 
\[K\geq\sum_iK_i\geq \sum_{i\in \cI_1}T_i+\sum_{i\in \cI_2}K_i\geq j+|\cI_2|(k+1),\]
%Now the remaining length can be at most $K - j$, 
so we have
\[
|\cI_2| \leq \Big\lfloor \frac{K - j}{k+1} \Big\rfloor.
\]
So the number of non-wildcards in $v$ is exactly
\[
|\cI_1| + |\cI_2| \leq \max_{0 \leq j \leq T}\left\{\Big\lfloor \frac{j}{a+1} \Big\rfloor + \Big\lfloor \frac{K - j}{k+1} \Big\rfloor \right\} \leq t
\]
by assumption.

To show the second part of the lemma, we can choose distinct $x, y \in \{0, 1\}^{S(k, a)}$ such that $x \keq{k}{a} y$ and distinct $p, q \in \{0, 1\}^{S(K, t)}$ such that $p \keq{K}{t} q$. Then $p(x, y) \neq q(x, y)$ and, by the first part of the lemma, $p(x, y) \keq{K}{T} q(x, y)$. Thus, there exist two distinct $(K,T)$-equivalent words of length $S(k,a)S(K,t)$, showing that $S(K,T) \leq S(k,a)S(K,t)$.
\end{proof}
%--------------------------------%

A direct application of the lemma gives \cref{thm:weak-deck-upper} (restated below for convenience), which improves the bound
\begin{equation}\label{eq:DS-weak-upper}
    S(k, t) \leq O_t \left(k^{t+1} \log k \right)
\end{equation}
in \cite[Lemma 8]{Dudik2003-fz}.

%--------------------------------%
\thmweakdeckupper*

\begin{proof}
Apply \cref{lem:template-lem} with $K =k$. With this choice of $K$, the second floor term of the condition on $t$ in the lemma vanishes. So when 
\[T \leq (a + 1) (b + 1) - 1,\]
the second part of \cref{lem:template-lem} gives 
\begin{equation}
\label{eq:templaterecursion}
S(k,T) \leq S(k,a) S(k,b) . 
\end{equation}
Now, let $r = \lceil \log_3 (t + 1) \rceil$,
and for $i\in [r]$ define 
\[T_i := \min\{ t, 3^{i} -1 \}. \]
Then, $T_r = t$ and $T_i  \leq  3 ( T_{i-1} + 1) -1 $. Now, we can apply $\eqref{eq:templaterecursion}$ with $a = T_{i-1}$ and $b=2$ to get 
\[ S(k,T_i ) \leq S(k,T_{i-1}) S(k,2) , \quad i \geq 2.\]
Iterating this bound from $i=r$ to $i=2$, and using the upper bound \eqref{eq:DS-weak-upper}, we obtain, for a fixed constant $C$
\[S(k,t) \leq S(k,2)^r \leq (Ck^3 \log k)^r = C^rk^{3 \lceil\log_3 (t + 1 ) \rceil } (\log k)^{\lceil\log_3 (t+1) \rceil} = k^{\frac{3}{\log 3} \log (t+1) + O(\log\log k)+O(\frac{\log t}{\log k})}. \qedhere \]

\end{proof}
%--------------------------------%

\end{subsection}
%----------------------------------------------------------%

%----------------------------------------------------------%
\begin{subsection}{Template Generated Classes}\label{sec:template-classes}

A natural question in the setting of Lemma \ref{lem:template-lem} is when the templates and substitutions preserve equivalence. The following result, restated from the introduction for convenience, provides a sufficient condition for this question.

%--------------------------------%
\thmtemplateequivalence*
%--------------------------------%

We start by giving some remarks and applications of the theorem before proving it. The conditions can be interpreted as the substitutions $x, y$ or the templates $p, q$ are far from being equivalent, since they do not even have the same proportion of 0s and 1s. Therefore, the theorem is saying that Lemma \ref{lem:template-lem} fails to construct a higher equivalence pair when the templates or the substitutions pairs are very different from each other. However, a more powerful use of it is to generate more classes of words that have same in-class reconstruction thresholds as much shorter words.

Consider the class
\[
\cT_d \defined \set{ x : x \text{ is a periodic binary word with period } d}.
\]
Equivalently, if $x \in \cT_d$, then $x = z^R$ for some $R \geq 1$ and $z \in \{0, 1\}^d$. The second statement of Theorem \ref{thm:template-equivalence} then gives us the following result on the in-class reconstruction threshold $r'(\cT_d)$.

%--------------------------------%
\corperiodic*

\begin{proof}
For $R \geq 1$, let $p = 0^R$ and $q = 1^R$; clearly $N(1, p) \neq N(1, q)$. Applying (2) in \cref{thm:template-equivalence}, for any $x, y \in \{0, 1\}^d$, we have
\[
x \keq y \iff x^R \keq y^R.
\]
Therefore, we see that distinguishing any pair of $d$-periodic words can be reduced to distinguishing two length $d$ words. So,
\[
r'(\cT_d) = r(d). \qedhere
\]
\end{proof}
%--------------------------------%

The following corollaries characterize word classes that preserve in-class reconstruction thresholds by applying \cref{thm:template-equivalence}. \cref{cor:template-family} is obtained by applying the second statement, which generalizes the setting in \cref{cor:periodic}, and \cref{cor:sub-family} is obtained by applying the first statement.

%--------------------------------%
\begin{corollary}\label{cor:template-family}
Let $p, q \in \{0, 1\}^+$, and let $\cF, \cG_{\cF} \subset \{0, 1\}^+$ be word classes. Suppose the following hold:
\begin{enumerate}
    \item For every distinct $x, y \in \cF$, there exist $p, q \in \{0, 1\}^+, \, |p| = |q|, \, N(1, p) \neq N(1, q)$, such that $p(x, y), q(x, y) \in \cG_{\cF}$.

    \item For every distinct $x', y' \in \cG_{\cF}$, there exist $p, q \in \{0, 1\}^+, \, |p| = |q|, \, N(1, p) \neq N(1, q)$, and distinct $x, y \in \cF$ such that $x' = p(x, y), y' = q(x, y)$.
\end{enumerate}
Then,
\[
r'(\cG_{\cF}) = r'(\cF).
\]
\end{corollary}

\begin{proof}
We just need to show that for any $k \geq 1$, there are distinct $x, y \in \cF$ where $x \keq y$ if and only if there are distinct $x', y' \in \cG_{\cF}$ where $x' \keq y'$.

Suppose $x \keq y$. Then by assumption 1, we have $p(x, y), q(x, y) \in \cG_{\cF}$ for some $p, q \in \{0, 1\}^+, |p| = |q|, N(1, p) \neq N(1, q)$. So, by (2) in \cref{thm:template-equivalence}, we have $p(x, y) \keq q(x, y)$.

Now suppose $x' \keq y'$. Then by assumption 2, there are $x, y \in \cF$, same length $p, q$ and $N(1, p) \neq N(1, q)$ where $x' = p(x, y)$ and $y' = q(x, y)$. Again by (2) in \cref{thm:template-equivalence}, we have $x \keq y$.
\end{proof}
%--------------------------------%

%--------------------------------%
\begin{corollary}\label{cor:sub-family}
Let $x, y \in \{0, 1\}^+$. Let $\cF \subset \{0, 1\}^+$ be a word class, and define
\begin{align*}
\cG_{\cF}^{xy} 
&:=
\{p(x, y) : p \in \cF\}.
\end{align*}
Then, if $\frac{N(0, x)}{N(1, x)} \neq \frac{N(0, y)}{N(1, y)}$, we have
\[
r'(\cG_{\cF}^{xy}) = r'(\cF).
\]
\end{corollary}

\begin{proof}
By construction, we have that for $p, q \in \{0, 1\}^+$,
\[
p, q \in \cF \iff p(x, y), q(x, y) \in \cG_{\cF}^{xy}.
\]
Since $\frac{N(0, x)}{N(1, x)} \neq \frac{N(0, y)}{N(1, y)}$, Theorem \ref{thm:template-equivalence} yields
\[
p \keq q \iff p(x, y) \keq q(x, y).
\]
Therefore, distinguishing any pair in $\cG_{\cF}^{xy}$ can be reduced to distinguishing a pair in $\cF$, so $r'(\cG_{\cF}^{xy}) = r'(\cF).$
\end{proof}
%--------------------------------%

To give an example, consider the class
\[
\cF_{\text{AP}} := \{w \in \{0, 1\}^+ : \supp(x) \text{ is an arithmetic progression}\}.
\]
Then for some fixed $L \geq 2$, by letting $x = 0, \, y = 1^L 0$, we obtain $\cG^{0, 1^L 0}_{\cF_{\text{AP}}}$, which can be interpreted as the class of binary strings that contain blocks of length $L$ whose starting positions form an arithmetic progression. By \cref{cor:sub-family}, we have
\[
r'(\cG^{0, 1^L 0}_{\cF_{\text{AP}}}) = r'(\cF_{\text{AP}}).
\]

We will now prove \cref{thm:template-equivalence}. From the proof, we can also see that assumption 1 in \cref{cor:template-family} can be slightly weakened by removing the requirement $N(1, p) \neq N(1, q)$.

%--------------------------------%
\begin{proof}[Proof of \cref{thm:template-equivalence} (1)]
For this proof, it will be easier to work with matrix $C_{xy} \in \NN^{2 \times 2}$ defined as
\[
C_{xy} := 
\begin{bmatrix}
N(0, x) & N(0, y) \\
N(1, x) & N(1, y)
\end{bmatrix}.
\]
Note that the condition $\frac{N(0, x)}{N(1, x)} \neq \frac{N(0, y)}{N(1, y)}$ is equivalent to $\det C_{xy} \neq 0$.

Suppose that $p \keq q$. For any $w \in \{0, 1\}^{\leq k}$, we have
\begin{equation}\label{eq:sub-count}
N(w, p(x, y))
=
\sum_{m = 1}^{|w|}\sum_{w = w^{(1)} \dots w^{(m)}} \sum_{s \in \{0, 1\}^m} \left(\prod_{j = 1}^{m} N(w^{(j)}, s_j(x, y)) \right)N(s, p),
\end{equation}
where $w = w^{(1)} \dots w^{(m)}$ is a partition of $w$ into $m$ nonempty consecutive blocks and $s_j$ is the $j$th bit in $s$. Similar to the counting argument for \eqref{eq:template-count} in \cref{lem:template-lem}, we can justify the above equation. Specifically, for some fixed $m$ and fixed partition $w = w^{(1)} \dots w^{(m)}$, we let a strictly increasing map $u:[m] \to [|p|]$ be a partition embedding. For this particular embedding, we can associate a ``pattern'' $s \in \{0, 1\}^m$, where $s_j = p_{u(j)}$ for $1 \leq j \leq m$. Then, the number of copies of $w$ that can appear in $p(x, y)$ under this particular partition embedding $u$ is $\prod_{j = 1}^{m} N(w^{(j)}, s_j(x, y))$. For each pattern $s \in \{0, 1\}^m$, $N(s, p)$ is exactly the number of partition embeddings $u$ satisfying $p_{u(j)} = s_j$ for all $j$. So, summing over all possible ways to partition $w$ and all the possible patterns yields the desired count for $N(w, p(x, y))$.

Since $p \keq q$, we have $N(s, p) = N(s, q)$ for all $|s| \leq k$, so $N(w, p(x, y)) = N(w, q(x, y))$. Therefore, $p(x, y) \keq q(x, y)$.

Now suppose that $p(x, y) \keq q(x, y)$ and $\det C_{xy} \neq 0$. For some $l \geq 1$, define the matrix $A_{x, y, l} \in \NN^{2^l \times 2^l}$, indexed by $w, s \in \{0, 1\}^l$ in some fixed ordering, by
\[
\left(A_{x, y, l}\right)_{w, s} := \prod_{j = 1}^{l} N(w_j, s_j(x, y)).
\]
More explicitly, $\left(A_{x, y, l}\right)_{w, s}$ is the number of copies of $w$ in $s(x,y)$ with each bit mapped to a different blow-up block.

For a word $z$, let $D_k(z)$ be the $k$-deck of $z$ in vector form, $D_k(z)=(N(w,z))_{w\in \{0,1\}^k}\in \mathbb N^{2^k}$. Then for $w\in\{0,1\}^l$, by pulling out the partition of singletons where each bit maps to a different block, we have
\[
N(w, p(x, y))
=
\sum_{s \in \{0, 1\}^l} \left(A_{x, y, l}\right)_{w, s} N(s, p) 
+
\sum_{m = 1}^{l - 1}\sum_{w = w^{(1)} \dots w^{(m)}} \sum_{s \in \{0, 1\}^m} \left(\prod_{j = 1}^{m} N(w^{(j)}, s_j(x, y)) \right)N(s, p)
\]
\[
=
(A_{x,y,l}D_l(p))_w
+
R_{x,y,l}(p)_w,
\]
where the first term is a matrix multiplication and the second term $R_{x,y,l}(p)_w$ above only involves $N(s, p)$ where $|s| < l$.
We can similarly write $N(w, q(x, y))$ with $p$ replaced with $q$. In short, for all $1 \leq l \leq k$, we have
\[
D_l(p(x, y)) = A_{x, y, l} D_{l}(p) + R_{x, y, l}(p),
\]
where $R_{x, y, l}(p)$ only depends on $D_j(p)$ with $j < l$. 

Now, we show by induction on $|w|$  that if $C_{x, y}$ is invertible, then $N(w, p) = N(w, q)$ for all $w \in \{0, 1\}^{\leq k}$, i.e. $D_l(p)=D_l(q)$ for all $l\leq k$. The base case is
\[
C_{x, y}D_1(p) = D_1(p(x, y)) = D_1(q(x, y)) = C_{x, y}D_1(q).
\]
Now suppose that $D_j(p) = D_j(q)$ for all $j < l \leq k$, so in particular $R_{x,y,l}(p)=R_{x,y,l}(q)$. We then have 
\begin{align*}
A_{x, y, l}D_l(p)
&=
D_l(p(x, y)) - R_{x, y, l}(p) \\
&=
D_l(q(x, y)) - R_{x, y, l}(q) \\
&=
A_{x, y, l}D_l(q).
\end{align*}
Observe that $A_{x, y, l} = C_{xy}\otimes \cdots\otimes C_{xy}=C^{\otimes l}_{x, y}$, where $\otimes$ is the Kronecker product, so $A_{x,y,l}$ is invertible for all $l\geq 1$. We conclude that $D_l(p) = D_l(q)$. So, we have $D_k(p) = D_k(q)$, implying that $p \keq q$. This finishes the proof of (1).
\end{proof}

\begin{proof}[Proof of \cref{thm:template-equivalence} (2)]
Suppose that $x \keq y$ and $|p| = |q|$. Let $\cP_{w, m}$ be the set of all possible partitions of $w$ into $m$ nonempty consecutive blocks. For $s \in \{0, 1\}^m$ and $\cP \in \cP_{w, m}$, define
\[
N'_{w, {\cP}} := \prod_{j = 1}^{m} N(w^{(j)}, s_j(x, y)).
\]
Since $x \keq y$, $N'_{w, \cP}$ is a constant that depends on $w$ and the partition $\cP$. So, we can rewrite \eqref{eq:sub-count} as
\begin{align*}
N(w, p(x, y))
&=
\sum_{m = 1}^{|w|}\sum_{\cP \in \cP_{w, m}} N'_{w, {\cP}} \sum_{s \in \{0, 1\}^m} N(s, p) \\
&=
\sum_{m = 1}^{|w|}\sum_{\cP \in \cP_{w, m}} N'_{w, {\cP}} \binom{|p|}{m}.
\end{align*}
We have a similar expression for $N(w, q(x, y))$ by replacing $|p|$ above with $|q|$. So, if $|p| = |q|$, we see that $N(w, p(x, y)) = N(w, q(x, y))$ for every $w \in \{0, 1\}^{\leq k}$, implying that $p(x, y) \keq q(x, y)$.

To prove the converse, we use a similar induction idea as above. Suppose that $p(x, y) \keq q(x, y)$, $|p| = |q|$, and $N(1, p) \neq N(1, q)$. Let $l \geq 1$, and split \eqref{eq:sub-count} into
\begin{align}
N(w, p(x, y))
&=
\sum_{s \in \{0, 1\}} N(w, s(x, y)) N(s, p)
+
\sum_{m = 2}^{l}\sum_{w = w^{(1)} \dots w^{(m)}} \sum_{s \in \{0, 1\}^m} \left(\prod_{j = 1}^{m} N(w^{(j)}, s_j(x, y)) \right)N(s, p) \notag \\
&=
\sum_{s \in \{0, 1\}} N(w, s(x, y)) N(s, p) + R_{x, y, l}(p). \label{eq:sub-induct}
\end{align}
Observe that $R_{x, y, l}(p)$ only contains partitioned blocks of $w$ with length strictly less than $l$. Let $p_0 = N(0, p)$ and $q_0 = N(0, q)$, and let $n = |p| = |q|$. Notice that the first sum in \eqref{eq:sub-induct} can be written as
\[
p_0 N(w, x) + (n - p_0) N(w, y).
\]

We show by induction on $|w|$ that $N(w, x) = N(w, y)$ for all $w \in \{0, 1\}^{\leq k}$. 

If $|w| = 1$, we have
\[
p_0 N(w, x) + (n - p_0) N(w, y)
=
q_0 N(w, x) + (n - q_0) N(w, y),
\]
which is equivalent to
\[
(p_0 - q_0)(N(w, x) - N(w, y)) = 0.
\]
Since $N(1, p) \neq N(1, q)$, we have $p_0 - q_0 \neq 0$, so $N(w, x) = N(w, y)$. Suppose that for some $l \leq k$, $N(w', x) = N(w', y)$ for all $|w'| < l$. Then, since $|p| = |q|$, we have
\begin{align*}
R_{x, y, l}(p) 
= 
\sum_{m = 2}^{l}\sum_{\cP \in \cP_{w, m}} N'_{w, {\cP}} \binom{|p|}{m}.
\end{align*}
The same argument as above yields $R_{x, y, l}(p) = R_{x, y, l}(q)$. So, we again have $(p_0 - q_0)(N(w, x) - N(w, y)) = 0$, which implies $N(w, x) = N(w, y)$ for $|w| = l$. Therefore, $x \keq y$, which concludes the proof.
\end{proof}

%--------------------------------%

\end{subsection}
%----------------------------------------------------------%

\end{section}
%----------------------------------------------------------%

%----------------------------------------------------------%
\begin{section}{Conclusion}

We obtained reconstruction bounds for several restricted classes of binary words and showed that these bounds depend only on the structural parameters of the words class rather than the word length. We also provide a generalization of the template lemma given by Dudík and Schulman that improves the upper bound on weak deck reconstruction and allows us to characterize when template substitutions preserve $k$-equivalence.

Several questions remain open. For the general deck reconstruction problem, the known bounds are $e^{\Omega(\sqrt{\log n})} \leq r(n) \leq O(\sqrt{n})$ \cite{Dudik2003-fz, Krasikov1997-rp}, and it is unclear whether the upper or lower bound is closer to the true asymptotic behavior of $r(n)$. Future research may also investigate other restricted classes of binary words and find bounds on their reconstruction thresholds based on their structural properties. Finally, we raise an algorithmic question: is there an efficient algorithm which takes as input the $k$-deck of a word, where $k$ is large enough that this uniquely determines the word, and outputs the word?

\subsection{AI use statement}
A discussion with ChatGPT 5.5 resulted in the matrix interpretation of the counting formula in the proof of Theorem \ref{thm:template-equivalence} (1), which enabled the authors to obtain the precise condition needed for the induction argument in that proof. AI was not used to help develop any other results or proofs in this paper.
\end{section}
%----------------------------------------------------------%

\printbibliography

@ARTICLE{Krasikov1997-rp,
  title    = "On a reconstruction problem for sequences",
  author   = "Krasikov, Ilia and Roditty, Yehuda",
  journal  = "Journal of Combinatorial Theory, Series A",
  volume   = 77,
  number   = 2,
  pages    = "344--348",
  month    = feb,
  year     = 1997,
  doi      = "10.1006/jcta.1997.2732",
  issn     = "0097-3165",
  url      = "https://www.sciencedirect.com/science/article/pii/S0097316597927321",
  language = "en"
}

@ARTICLE{Dudik2003-fz,
  title     = "Reconstruction from subsequences",
  author    = "Dudı́k, Miroslav and Schulman, Leonard J",
  journal   = "J. Comb. Theory Ser. A.",
  publisher = "Elsevier BV",
  volume    =  103,
  number    =  2,
  pages     = "337--348",
  month     =  aug,
  year      =  2003,
  language  = "en"
}

@ARTICLE{Kalashnik1973,
  title     = "The reconstruction of a word from fragments",
  author    = "Kalashnik, L. I.",
  journal   = "Numerical Mathematics and Computer Technology",
  publisher = "Akad. Nauk Ukrain. SSR Inst. Mat.",
  note      = "Preprint IV",
  pages     = "56--57",
  year      =  1973
}

@ARTICLE{Scott1997-mh,
  title     = "Reconstructing sequences",
  author    = "Scott, A D",
  journal   = "Discrete Math.",
  publisher = "Elsevier BV",
  volume    =  175,
  number    = "1-3",
  pages     = "231--238",
  month     =  oct,
  year      =  1997,
  language  = "en"
}

@ARTICLE{Kelly1957,
  title     = "A congruence theorem for trees",
  author    = "P. J. Kelly",
  journal   = "Pacific J. Math",
  publisher = "Wiley",
  volume    =  7,
  number    =  1,
  pages     = "961-968",
  month     =  jan,
  year      =  1957,
  language  = "en"
}

@phdthesis{Kelly1942,
    author = "P. J. Kelly",
    title = "On isometric transformations",
    school = "University of Wisconsin-Madison",
    year = 1942
}

@ARTICLE{Batu2004-pq,
  title     = "Reconstructing strings from random traces",
  author    = "Batu, Tuǧkan and Kannan, Sampath and Khanna, Sanjeev and
               McGregor, Andrew",
  journal = "Proceedings of the fifteenth annual ACM-SIAM symposium on
               Discrete algorithms",
  publisher = "Society for Industrial and Applied Mathematics",
  address   = "USA",
  pages     = "910--918",
  series    = "SODA '04",
  month     =  jan,
  year      =  2004,
  language  = "en"
}

@ARTICLE{Raykova2006-jl,
  title     = "Permutation Reconstruction from Minors",
  author    = "Raykova, Mariana",
  journal   = "Electron. J. Comb.",
  publisher = "The Electronic Journal of Combinatorics",
  volume    =  13,
  number    =  1,
  pages     = "R66",
  month     =  aug,
  year      =  2006,
  language  = "en"
}

@ARTICLE{Manvel1991-mh,
  title     = "Reconstruction of sequences",
  author    = "Manvel, B and Meyerowitz, A and Schwenk, A and Smith, K and
               Stockmeyer, P",
  journal   = "Discrete Math.",
  publisher = "Elsevier BV",
  volume    =  94,
  number    =  3,
  pages     = "209--219",
  month     =  nov,
  year      =  1991,
  language  = "en"
}

@INPROCEEDINGS{Gabrys2017-ju,
  title     = "The hybrid k-deck problem: Reconstructing sequences from short
               and long traces",
  author    = "Gabrys, Ryan and Milenkovic, Olgica",
  booktitle = "2017 IEEE International Symposium on Information Theory (ISIT)",
  publisher = "IEEE",
  pages     = "1306--1310",
  month     =  jun,
  year      =  2017,
  language  = "en"
}

@INPROCEEDINGS{Golm2022-jy,
  title     = "The Gapped k-Deck Problem",
  author    = "Golm, Rebecca and Nahvi, Mina and Gabrys, Ryan and Milenkovic,
               Olgica",
  booktitle = "2022 IEEE International Symposium on Information Theory (ISIT)",
  publisher = "IEEE",
  pages     = "49--54",
  month     =  jun,
  year      =  2022
}

@ARTICLE{Yehezkeally2023-jm,
  title     = "Generalized Unique Reconstruction From Substrings",
  author    = "Yehezkeally, Yonatan and Bar-Lev, Daniella and Marcovich, Sagi
               and Yaakobi, Eitan",
  journal   = "IEEE Trans. Inf. Theory",
  publisher = "Institute of Electrical and Electronics Engineers (IEEE)",
  volume    =  69,
  number    =  9,
  pages     = "5648--5659",
  month     =  sep,
  year      =  2023
}

@ARTICLE{Chase2020-if,
  title         = "New upper bounds for trace reconstruction",
  author        = "Chase, Zachary",
  journal       = "arXiv [math.PR]",
  month         =  sep,
  year          =  2020,
  archivePrefix = "arXiv",
  primaryClass  = "math.PR"
}

@ARTICLE{Ma2024-sk,
  title         = "Improvements on permutation reconstruction from minors",
  author        = "Ma, Yiming and Zhong, Wenjie and Zhang, Xiande",
  journal       = "arXiv [math.CO]",
  month         =  nov,
  year          =  2024,
  archivePrefix = "arXiv",
  primaryClass  = "math.CO"
}

%----------------------------------------------------------%
\appendix
\begin{section}{Appendix: Computations for \Cref{subsec:fixedweight}}\label{sec:appendix}
 This section provides the computations that lead to equations \eqref{eq:g2identity}, \eqref{eq:g3identity}, and\eqref{eq:delta0identity}. For the remainder of the section, we let $s \geq 5$ and 
\[x = 0^{g_0}10^{g_1} \dots 10^{g_s}, \quad y = 0^{g'_0}10^{g'_1} \dots 10^{g'_s}\]
be words with $s$ ones such that $x \keq{s} y $.  Let $\delta =g-g'$ and recall that equation \eqref{eq:gapidentity} gives 
 \[\delta_i = (-1)^i { s \choose i} \delta_0.\]
 For convenience, we denote $G_j = \sum_{i=0}^j g_i$. Then, 
 \[G'_j = G_j + \alpha_j \delta_0, \quad \alpha_j = \sum_{i = 0}^j (-1)^i {s \choose i} = (-1)^j {s -1 \choose j}. \]
\end{section}

%----------------------------------------------------------%
\begin{subsection}{Equation for \texorpdfstring{$g_2$}{g2}}
\begin{lemma} Under the above assumptions, 
    \[ g_2 = g_1 - \frac{s(s+1)}{4} \delta_0 . \]
\end{lemma}
\begin{proof}
Let $q = { s -1 \choose 2}$ and let $(w_0,w_1, w_2) = (q,s-2,1)$, so we have
\[
N(001^{s-2},x)=w_0\binom{G_0}{2}+w_1\binom{G_1}2+w_2\binom{G_2}2.
\]
Since $(\alpha_0, \alpha_1, \alpha_2) = (1,1-s, q),$ we have the identities
\[\sum_{j = 0}^2 w_j \alpha_j = 0, \quad \sum_{j=0}^2 w_j \alpha_j G_j = q(G_0 -2 G_1 + G_2)  = q(g_2 - g_1) ,\]
\[\sum_{j =0}^2 w_j \alpha_j^2 = q + (s-2) (s-1)^2 + q^2 = \frac{qs(s+1)}{2}.\]
With this, and using ${z + t \choose 2}- {z \choose 2} = tz + (t^2 - t) /2$ in the two expansions of $N(001^{s-2}, \cdot )$, gives 
\begin{align*}
    0 & = \sum_{j = 0}^2 w_j \left[ {G_j + \alpha_j \delta_0 \choose 2} - {G_j \choose 2} \right] \\
    & = q\delta_0 (g_2 - g_1) + \frac{qs(s+1)}{4} \delta_0^2 = q\delta_0 \left(g_2 - g_1 + \frac{s(s+1)}{4} \delta_0\right).
\end{align*}
Since we assumed $\delta_0 \neq 0,$ we have
\[ g_2 = g_1 - \frac{s(s+1)}{4} \delta_0, \] as desired.
\end{proof}
\end{subsection}
%----------------------------------------------------------%

%----------------------------------------------------------%
\begin{subsection}{Equation for  \texorpdfstring{$g_3$}{g3}}
\label{subsec:g3identityappendix}
\begin{lemma} Under the above assumptions, 
    \[ g_3 = g_1 + \frac{s(s+1 ) (s-4)}{12} \delta_0. \]
\end{lemma}
\begin{proof}
Now we let $w_j = {s - j -1 \choose 3 - j}$, so 
\[
N(0001^{s-3},x)= w_0\binom{G_0}{3}+w_1\binom{G_1}{3}+w_2\binom{G_2}{3}+w_3\binom{G_3}{3}.
\] 
Recalling that $G'_j = G_j + \alpha_j\delta_0$, we have
\[w_j \alpha_j = (-1)^j {s - j -1 \choose 3 - j} {s -1 \choose j} = (-1)^j \frac{(s-1)!}{j! (3-j)! (s-4)!} = (-1)^j{3 \choose j} {s -1 \choose 3 }.\]
Thus $(w_0 \alpha_0,w_1\alpha_1, w_2\alpha_2, w_3\alpha_3 ) = {s-1 \choose 3} (1,-3,3,-1)$ and 
\[\sum_{j=0}^3 w_j\alpha_j = 0.\]
Now, the equality of the two counts give together with ${z + t \choose 2} - {z \choose 2} = tz + \frac{t^2-t}{2}$ gives
\begin{align*}
    0 & = \sum_{j = 0}^3 w_j \left[ {G_j + \alpha_j \delta_0 \choose 2} - {G_j \choose 2}\right] \\
 & = \delta_0 \sum_{j =0}^3 w_j \alpha_j G_j + \frac{\delta_0^2}{2}\sum_{j=0}^3 w_j \alpha_j^2 - \frac{\delta_0}{2} \sum_{j=0}^3 w_j \alpha_j . 
\end{align*}
The last term vanishes, and dividing by ${s - 1 \choose 3} \delta_0$ leaves 
\[ 0 = (G_0 - 3G_1 + 3G_2 - G_3 ) + \frac{\delta_0}{2} (\alpha_0 - 3 \alpha_1 + 3\alpha_2 - \alpha_3).\]
We can simplify the two parenthesis separately. For the first, 
\begin{align*}
G_0 - 3G_1 + 3G_2 -G_3 & =-(G_1-G_0) + 2(G_2-G_1) -(G_3-G_2) \\
& = -g_1 + 2g_2 - g_3.
\end{align*}
For the second, direct substitution gives 
\begin{align*}
\alpha_0 - 3 \alpha_1 + 3\alpha_2 - \alpha_3 & = 1 + 3 {s-1 \choose 1} + 3 {s-1 \choose 2} +  1 {s-1 \choose 3} \\
& = {s + 2 \choose 3}.\end{align*}
Combining everything, we have 
\[0 =-g_1 + 2g_2 -g_3 + \frac{\delta_0}{2} {s + 2 \choose 3}.\]
Substituting $g_2 = g_1 - s(s+1) \delta_0 /4$ from \eqref{eq:g2identity} gives
\begin{align*}
g_3 & = g_1 - \frac{s(s+1)}{2} \delta_0 + \frac{s(s+1)(s+2) }{12} \delta_0 \\
      & = g_1 + \frac{s(s+1 ) (s-4)}{12} \delta_0,
\end{align*} as desired.
\end{proof}
\end{subsection}
%----------------------------------------------------------%

%----------------------------------------------------------%
\begin{subsection}{Equation for  \texorpdfstring{$\delta_0$}{delta0}}
\begin{lemma} Under the above assumptions, 
    \[{s(s-1)(s-2)(s-3)(s-4)(s+1)(s^3 - 9s^2 + 8s+6)}\delta_0^3 = 0.\]
\end{lemma}
\begin{proof}
Recall $G_j = \sum_{i=0}^j g_i$ and $\alpha_j \delta_0= \sum_{i = 0}^j \delta_i $. Equations \eqref{eq:g2identity} and \eqref{eq:g3identity} give $g_i + \delta_i / 2 = g_1 - s\delta_0 /2$ for $i = 1, 2,3.$ Adding these averaged gaps give 
\[G_j + \frac{\alpha_j \delta_0}{2} = g_0 + \frac{\delta_0}{2} + j \left( g_1 - \frac{s\delta_0}{2}\right),\]
for $0 \leq j \leq 3$. Let $w_j = {s - j -1 \choose 3 -j}$ as in  \Cref{subsec:g3identityappendix}. Recall, $(w_0 \alpha_0,w_1\alpha_1, w_2\alpha_2, w_3\alpha_3 ) = {s-1 \choose 3} (1,-3,3,-1)$. These coefficients annihilate constants, linear terms, and squares as 
\[1-3 + 3 -1 = 0 , \quad 1(0) -3(1) +3(2) -1(3) = 0 , \quad  1(0) -3(1) +3(4) -1(9) = 0. \]
Since $G_j + \alpha_j\delta_0 /2$ is linear in $j$, the following sums vanish
\[\sum_{j = 0}^3 w_j \alpha_j = 0 , \quad \sum_{j=0}^3 w_j \alpha_j \left(G_j - 1 + \frac{\alpha_j\delta_0}{2}\right)=0.\]
Using this and the identity 
\[{z + t \choose 3} - {z \choose 3} = \frac{t}{2}\left( z -1 + \frac{t}{2} \right)^2 + \frac{t^3}{24}- \frac{t}{6},\]
in our equality of counts gives 
\begin{align*}
    0 & = \sum_{j = 0}^3 w_j \left[ {G_j + \alpha_j \delta_0 \choose 3} - {G_j \choose 3}\right] = \frac{\delta_0^3}{24} \sum_{j=0}^3 w_j \alpha_j^3 \\
 & =  \frac{{s-1 \choose 3} \delta_0^3}{864} [36 - 108(s-1)^2 + 27(s-1)^2(s-2)^2 - (s-1)^2(s-2)^2(s-3)^2]. 
\end{align*}
Expanding ${s -1 } \choose 3$, factoring, and multiplying by $-5184$ gives
\[{s(s-1)(s-2)(s-3)(s-4)(s+1)(s^3 - 9s^2 + 8s+6)}\delta_0^3 = 0,\] as desired.
\end{proof}
\end{subsection}
%----------------------------------------------------------%

%----------------------------------------------------------%
\end{document}